\documentclass[12pt]{article}
\usepackage{amsmath, amsthm, amssymb, stmaryrd}
\usepackage[all]{xy}
\usepackage{fullpage}
\usepackage{titlesec}
\usepackage{esint}

\titleformat{\section}{\normalsize\bfseries}{\thesection}{1em}{}
\titleformat{\subsection}{\normalsize\bfseries}{\thesubsection}{1em}{}

\theoremstyle{plain}

\newtheorem{PropSub}[subsection]{Proposition}
\newtheorem{LemSub}[subsection]{Lemma}
\newtheorem{CorSub}[subsection]{Corollary}
\newtheorem{ThmSub}[subsection]{Theorem}

\theoremstyle{definition}

\newtheorem{DefSub}[subsection]{Definition}
\newtheorem{ExaSub}[subsection]{Example}
\newtheorem{RemSub}[subsection]{Remark}

\DeclareMathOperator{\Ev}{Ev}

\DeclareMathOperator{\pint}{\sqint}

\newcommand{\SM}{\ensuremath{S}}

\newcommand{\A}{\ensuremath{\mathcal{A}}}
\newcommand{\B}{\ensuremath{\mathcal{B}}}
\newcommand{\C}{\ensuremath{\mathcal{C}}}

\newcommand{\I}{\ensuremath{\mathcal{I}}}

\newcommand{\M}{\ensuremath{\mathcal{M}}}

\newcommand{\MM}{\ensuremath{\mathbb{M}}}
\newcommand{\LL}{\ensuremath{\mathbb{L}}}

\newcommand{\RR}{\ensuremath{\mathbb{R}}}
\newcommand{\NN}{\ensuremath{\mathbb{N}}}

\newcommand{\Born}{\ensuremath{\operatorname{\textnormal{\textbf{Born}}}}}
\newcommand{\BornMeas}{\ensuremath{\operatorname{\textnormal{\textbf{BornMeas}}}}}
\newcommand{\Meas}{\ensuremath{\operatorname{\textnormal{\textbf{Meas}}}}}
\newcommand{\Set}{\ensuremath{\operatorname{\textnormal{\textbf{Set}}}}}
\newcommand{\RVect}{\ensuremath{\operatorname{\textnormal{\textbf{$\RR$-Vect}}}}}
\newcommand{\Ban}{\ensuremath{\operatorname{\textnormal{\textbf{Ban}}}}}
\newcommand{\RConvBvs}{\ensuremath{\operatorname{\textnormal{\textbf{$\RR$-ConvBvs}}}}}

\newcommand{\subs}{\ensuremath{\subseteq}}

\newcommand{\Powset}{\ensuremath{\mathcal{P}}}

\newcommand{\id}{\ensuremath{\mathsf{id}}}

\newcommand{\lt}{\leqslant}
\newcommand{\gt}{\geqslant}

\begin{document}

\author{\normalsize  Rory B.B. Lucyshyn-Wright\thanks{Partial financial assistance by the Ontario Graduate Scholarship program is gratefully acknowledged.}\\
\small York University, 4700 Keele St., Toronto, ON, Canada M3J 1P3}

\title{\large \textbf{Algebraic theory of vector-valued integration}}

\date{}

\maketitle

\abstract{
We define a monad $\MM$ on a category of \textit{measurable bornological sets}, and we show how this monad gives rise to a theory of vector-valued integration that is related to the notion of \textit{Pettis integral}.  We show that an \textit{algebra} $X$ of this monad is a bornological locally convex vector space endowed with operations which associate vectors $\int f\;d\mu$ in $X$ to incoming maps $f:T \rightarrow X$ and measures $\mu$ on $T$.  We prove that a Banach space is an $\MM$-algebra as soon as it has a Pettis integral for each incoming bounded weakly-measurable function.  It follows that all separable Banach spaces, and all reflexive Banach spaces, are $\MM$-algebras.
}

\section{Introduction}

A fundamental paradigm of algebra has been the process of abstraction whereby the form and governing equations of the familiar operations of arithmetic have been isolated, yielding abstract notions, such as those of abelian group, ring, and vector space, of which the real numbers are an example in the company of others.  It is the aim of this paper to proceed analogously with regard to the operations $f \mapsto \int f \;d\mu$ of Lebesgue integration with respect to measures $\mu$.  We provide an equational axiomatization of such operations, thereby defining a general algebraic notion of a space in which integrals may take their values.

The usual real-valued Lebesgue integration can indeed be construed as a family of operations
$$\Omega^T_\mu : \RR^T \rightarrow \RR\;,\;\;f \mapsto \int f\;d\mu$$
associated to measurable spaces $T$ and measures $\mu$ thereon, where $\RR^T$ is a suitable set of real-valued functions $f:T \rightarrow \RR$, each of which may be regarded as a $T$-indexed family of points in $\RR$ to which the given $T$-ary operation may be applied.  We axiomatize a notion of a space $X$ equipped with an analogous family of operations $\Omega^T_\mu:X^T \rightarrow X$, again written as $f \mapsto \int f\;d\mu = \int_{t \in T}f(t) \;d\mu$, satisfying certain equations.

The presence of the integration operations carried by such a space $X$ entails in particular that $X$ will carry the structure of a vector space over the reals, even though our axiomatization does not directly mandate this.  Rather, linear combinations $a_1 x_1 + ... + a_n x_n$ of elements $x_i \in X$ may be taken by considering the discrete space $T := \{1,...,n\}$ and forming an integral
$$a_1 x_1 + ... + a_n x_n := \int_{t \in T} x_t \;d\mu$$
with respect to a corresponding linear combination $\mu := a_1 \delta_1 + ... + a_n \delta_n$ of \textit{Dirac measures} $\delta_t$.  We thus \textit{define} associated vector space operations in terms of the given operations of integration, and the equational laws which must be satisfied by these derived vector space operations are then entailed by those governing the integration operations.

Hence, we provide a fresh perspective on \textit{vector-valued integration}, a subject which was a central motivation for the development of Banach space theory in the 1930s \cite{DU}.  The subject reached an apparent apex of generality with the introduction of the \textit{Pettis integral} in 1938 \cite{Pe}, the modern understanding of which has flourished since the late 1970s and the work of Edgar and Talagrand; see \cite{T}.  We show that spaces having sufficient Pettis integrals, such as reflexive or separable Banach spaces, provide examples of our general algebraic notion.

Our equational axiomatization is achieved by defining a \textit{monad} $\MM$ which concisely and canonically captures the syntax and equational laws of our theory of vector-valued integration.  Monads are a cornerstone of the category-theoretic approach to Birkhoff's general algebra initiated by Lawvere \cite{La}.  Categories of finitary algebras can be presented elegantly and canonically both in terms of Lawvere's \textit{algebraic theories} and, alternatively, as the categories of Eilenberg-Moore algebras \cite{EM} of \textit{finitary monads} on the category $\Set$ of sets (see, e.g., \cite{PT} or \cite{ARV}).  The connection between algebraic theories and monads was elucidated by Linton \cite{Li1,Li2,Li3}, who showed that monads axiomatize not only finitary but also infinitary algebras.  Further, Linton showed that even monads on arbitrary abstract categories, rather than $\Set$, also give rise to categories of algebras defined by the operations and equations of an associated (generalized) algebraic theory.

In our case, the monad $\MM = (M,\delta,\kappa)$ is defined on a category $\BornMeas$ of sets $X$ equipped with both a sigma-algebra and a \textit{bornology}, which is a system of subsets of $X$ that are said to be \textit{bounded} (see, e.g., \cite{HN}).  The functor $M:\BornMeas \rightarrow \BornMeas$ associates to $X$ the set $MX$ of all finite signed measures which are, in a suitable sense, supported by a bounded subset of $X$.  An Eilenberg-Moore algebra $(X,c)$ of $\MM$ consists of an object $X \in \BornMeas$ together with a boundedness-preserving measurable map $c:MX \rightarrow X$ making certain diagrams commutative.  Such an Eilenberg-Moore algebra can be described alternatively as an algebra of the associated algebraic theory (of the type of Linton \cite{Li1,Li2,Li3}), and hence carries \textit{operations}
$$\Omega_\mu^T:\BornMeas(T,X) \rightarrow X$$
of each \textit{arity} $T \in \BornMeas$ associated to each measure $\mu \in MT$, given in terms of $c:MX \rightarrow X$ via $\Omega_\mu^T(f) = c \circ Mf(\mu)$ for each morphism $f:T \rightarrow X$ in $\BornMeas$.  These we construe as the operations of integration valued in $X$, defining
\begin{equation}\label{eq:integ_via_struct_map}\int f \;d\mu := \Omega_\mu^T(f) = c \circ Mf(\mu)\;.\end{equation}
Moreover, such an algebra also carries \textit{multiply-valued} operations
$$\Omega_\mu^T:\BornMeas(T,X) \rightarrow \BornMeas(S,X)\;,$$
for each $T,S \in \BornMeas$ and each $\BornMeas$-morphism $\mu_{(-)}:S \rightarrow MT$, which send each $f:T \rightarrow X$ to the $S$-indexed family of integrals $\int f \;d\mu_s$.

All these operations of integration $\Omega_\mu^T$ reduce to the single \textit{structure map} $c:MX \rightarrow X$, which sends each measure $\mu \in MX$ to the integral 
$$\int \id_X \;d\mu = \Omega_\mu^X(\id_X) = c(\mu)$$
of the identity map $\id_X:X \rightarrow X$ with respect to $\mu$.  Indeed, this is the case in our primordial example of an $\MM$-algebra, the real line $\RR$, as each integral $\int f \;d\mu$ of a real-valued map $f:X \rightarrow \RR$ reduces to an integral $\int \id_{\RR} \;d\,Mf(\mu)$ with respect to the \textit{direct-image measure} $Mf(\mu)$ on $\RR$.  Any power $\RR^n$ of $\RR$ is also an $\MM$-algebra, and, for example, if we let $\mu$ be the probability measure associated to a uniform distribution across a measurable subset $E \subs \RR^n$, then $c(\mu) = \int \id_{\RR^n}\;d\mu$ will be the \textit{geometric centre} of $E$; for an arbitrary probability measure $\mu$ on $\RR^n$, $c(\mu)$ is the \textit{barycentre} or \textit{centre of mass} of $\mu$.

It may be remarked that our syntactic theory of vector-valued integration actually incorporates, and depends upon, real-valued (Lebesgue) integration.  Indeed, $\RR$, endowed with the Lebesgue integral, is itself a \textit{free} algebra $\RR \cong M1$ of $\MM$.  This is analogous to the situation of the algebraic theory of vector spaces over $\RR$, which similarly depends upon the addition and (scalar) multiplication operations of $\RR$.

Our monad $\MM$ is a variation on the Giry-Lawvere monad of probability measures on the category of measurable spaces \cite{G}.  As we employ signed real-valued measures, rather than probability measures, and consequently must also introduce notions of boundedness via bornologies, the definition of $\MM$ and the proof that $\MM$ is a monad are much more involved than those pertaining to the Giry-Lawvere monad.  Some of the lemmas we employ are similar in form to those of \cite{G}, but their proofs are much more difficult, owing to these added concerns of signed measures and boundedness and the need to employ more subtle convergence theorems with regard to signed real-valued integrals.

Giry \cite{G} also introduced a monad of probability measures on \textit{Polish spaces} rather than measurable spaces, and Doberkat \cite{D} characterized the algebras of that monad as certain topological \textit{convex spaces} (i.e. spaces endowed with operations allowing the taking of \textit{convex} combinations) and observed the connection of the structure map of such an algebra to the notion of barycentre.  Doberkat's work was predated by the work of \'Swirszcz \cite{Sw1,Sw2} and Semadeni \cite{Se} with regard to a monad of \textit{regular Borel probability measures} on compact Hausdorff spaces.  \'Swirszcz showed that the algebras of that monad are \textit{compact convex sets} embeddable within locally convex topological vector spaces.  \'Swirszcz also observed the connection to the barycentre or \textit{centroid} of a probability measure.

The recent papers of Kock \cite{K1,K2} on a general framework for \textit{extensive quantities} via monads are also related to our work.  Working with an arbitrary commutative strong monad $\mathbb{T}$ on a cartesian closed category, Kock has independently employed the integral notation as in \eqref{eq:integ_via_struct_map} with respect to a $T$-algebra.  However, Kock employs this notation chiefly in the case of the free algebra $R := T1$, as Kock's aim is clearly to provide a framework for extensive quantities valued in a special object $R$ analogous to the real line.  The examples of monads considered in \cite{K1,K2} are substantially different from our monad $\MM$, and Kock does not construe such monads as providing a theory of vector-valued integration.

\section{The base category: Measurable bornological sets}

\begin{DefSub}
A \textit{bornology} on a set $X$ is an ideal $\B X$ in the powerset $(\Powset X,\subs)$ of $X$ whose union is the entire set $X$.  Hence a bornology is a collection of subsets of $X$, called the \textit{bounded subsets}, that is downward-closed with respect to the inclusion order $\subs$, closed under the taking of finite unions (and hence, in particular, contains $\emptyset$), and contains all singletons.  A \textit{basis} for a bornology on $X$ is an upward-directed subset $\C$ of $\Powset X$ whose union is all of $X$, and for any such basis $\C$, the collection of all subsets of sets in $\C$ constitutes a bornology on $X$, the \textit{bornology generated by} $\C$.  
\end{DefSub}

\begin{DefSub}
A \textit{bornological set} is a set $X$ equipped with a bornology $\B X$.  We denote by $\Born$ the category of bornological sets and \textit{bornological maps}, i.e. functions $f:X \rightarrow Y$ for which the image of any bounded subset of $X$ is a bounded subset of $Y$.
\end{DefSub}

\begin{DefSub}
A \textit{measurable space} is a set $X$ equipped with a sigma-algebra $\sigma X$.  We denote by $\Meas$ the category of measurable spaces and \textit{measurable maps}, i.e. functions $f:X \rightarrow Y$ for which the inverse image of any measurable subset of $Y$ is a measurable subset of $X$.
\end{DefSub}

\begin{DefSub}
A \textit{measurable bornological set} is a set $X$ equipped with both a bornology $\B X$ and a sigma-algebra $\sigma X$.  We denote by $\BornMeas$ the category of measurable bornological sets, with maps that are both measurable and bornological.
\end{DefSub}

\begin{DefSub} \label{def:topological}
Let $P:\A \rightarrow \Set$ be a faithful functor, which we shall view as providing each object $X$ of $\A$ with an \textit{underlying set}, again written $X$.  We identify each hom-set $\A(X,Y)$ with the associated subset of $\Set(X,Y)$.  We say that a family of morphisms $(f_i:X \rightarrow Y_i)_{i \in I}$ in $\A$ is an \textit{initial source} in $\A$ if for any incoming function $g:T \rightarrow X$ defined on the underlying set of an object $T$ of $\A$, $g$ is a morphism in $\A$ as soon as each composite $T \xrightarrow{g} X \xrightarrow{f_i} Y_i$ is a morphism in $\A$.  The dual notion is that of a \textit{final sink}.  A single morphism $f:X \rightarrow Y$ of $\A$ is \textit{initial} if it constitutes an initial source with one element.  We say $P$ is \textit{topological} if for every family $(Y_i)_{i \in I}$ of objects in $\A$, indexed by a class $I$, and every family $(f_i:X \rightarrow Y_i)_{i \in I}$ of morphisms in $\Set$, there is an associated object of $\A$ with underlying set $X$, again written as $X$, with respect to which $(f_i:X \rightarrow Y_i)_{i \in I}$ is an initial source in $\A$.  If $P$ is topological then it follows that $P$ also has the dual property, that of being \textit{cotopological}; see, e.g., \cite{AHS}.
\end{DefSub}

\begin{PropSub} \label{thm:initial_final}
The forgetful functors $\Born \rightarrow \Set$ and $\Meas \rightarrow \Set$ are topological.  In particular, we have the following:
\begin{enumerate}
\item A family of morphisms $(f_i:X \rightarrow Y_i)_{i \in I}$ in $\Born$ is an initial source iff $X$ carries the \emph{initial bornology}, wherein
$$B \subs X \;\text{is bounded} \;\;\Leftrightarrow\;\; \forall i \in I \;:\; f_i(B) \subs Y_i \;\text{is bounded}\;.$$
\item A family of morphisms $(f_i:X \rightarrow Y_i)_{i \in I}$ in $\Meas$ is an initial source iff $X$ carries the \emph{initial sigma-algebra}, generated by the inverse images $f_i^{-1}(F)$ with $i \in I$ and $F \subs Y_i$ measurable.
\item A morphism $i:A \rightarrow X$ in $\Meas$ is initial iff it carries the sigma-algebra $\{ i^{-1}(E) \mid E \subs X \;\text{measurable}\}$.
\end{enumerate}
\end{PropSub}

\begin{CorSub} \label{thm:bornmeas_topological}
The forgetful functor $\BornMeas \rightarrow \Set$ is topological.  Initial structures in $\BornMeas$ are gotten by equipping a set with the initial bornology and initial sigma-algebra.  Hence the categories $\BornMeas$, $\Born$, and $\Meas$ are complete and cocomplete, with limits (resp. colimits) formed by endowing the limit (resp. colimit) in $\Set$ with the initial (resp. final) structure.
\end{CorSub}

\begin{RemSub} \label{rem:direct_limit_bornology}
For a diagram $D:\I \rightarrow  \Born$ with $\I$ an upward-directed poset, the colimit $Y = \varinjlim_{i \in \I} Di$ in $\Born$ carries the \textit{direct limit bornology}, which has a basis consisting of the images $\lambda_i(B) \subs Y$ of bounded subsets $B \subs Di$ under the colimit injections $\lambda_i:Di \rightarrow Y$.
\end{RemSub}

\begin{RemSub}
Given a subset $A \subs X$ of a bornological set, measurable space, or measurable bornological set $X$, we implicitly endow $A$ with the initial bornology and/or sigma-algebra induced by the inclusion $\iota_{AX}:A \hookrightarrow X$.
\end{RemSub}

\begin{LemSub} \label{thm:BornMeas_functions_vector_space}
For each $X \in \BornMeas$ the set $\BornMeas(X,\RR)$ of all $\BornMeas$-morphisms $X \rightarrow \RR$ is a real vector space under the pointwise operations.
\end{LemSub}

\section{Finite signed measures}

\begin{DefSub} \label{def:SX}
For each measurable space $X$, we let $\SM X$ be the set of all finite signed measures on $X$, and we endow $\SM X$ with the initial sigma-algebra induced by the \textit{evaluation maps}
$$\Ev_E:\SM X \rightarrow \RR\;,\;\mu \mapsto \mu(E)$$
associated to measurable subsets $E \subs X$.  There is a functor $\SM:\Meas \rightarrow \Meas$ which associates to each measurable map $f:X \rightarrow Y$ the map $\SM f:\SM X \rightarrow \SM Y$ sending each measure $\mu \in \SM X$ to the \textit{direct image} $\SM f(\mu) \in \SM Y$ \textit{of $\mu$ along $f$}, given by $(\SM f(\mu))(F) = \mu(f^{-1}(F))$ for each measurable $F \subs Y$.  Indeed, $\SM f$ is measurable since each composite
$$\SM X \xrightarrow{\SM f} \SM Y \xrightarrow{\Ev_F} \RR\;,\;\;\;\;\text{$F \subs Y$ measurable}\;, $$
is equal to the measurable map $\Ev_{f^{-1}(F)}:\SM X \rightarrow \RR$.
\end{DefSub}

\begin{PropSub}
For each measurable space $X$, $\SM X$ is a real vector space under the setwise operations, and for each measurable map $f:X \rightarrow Y$, $\SM f:\SM X \rightarrow \SM Y$ is a linear map.  Hence we obtain a functor $\Meas \rightarrow \RVect$.
\end{PropSub}

\begin{DefSub} \label{def:total_var}
Let $X$ be a measurable space.  For each $\mu \in \SM X$, we denote by $(\mu^+,\mu^-)$ the Jordan decomposition of $\mu$.  The \textit{total variation} of $\mu$ is the real number
$$||\mu|| = |\mu|(X)\;,$$
where $|\mu| = \mu^+ + \mu^- \in \SM X$ is the \textit{total variation measure} associated to $\mu$.
\end{DefSub}

\begin{RemSub} \label{rem:total_var_norm}
For a measurable space $X$, Definition \ref{def:total_var} endows the real vector space $\SM X$ with a norm, the \textit{total variation norm}, under which $\SM X$ is a Banach space (e.g., by Exercise 1.28 of \cite{Me}).
\end{RemSub}

\begin{PropSub} \label{thm:N_yields_contractions}
For each measurable map $f:X \rightarrow Y$ and each $\mu \in \SM X$, we have
$$||\SM f(\mu)|| \lt ||\mu||\;,$$
so that $\SM f:\SM X \rightarrow \SM Y$ is a linear contraction.  Hence we obtain a functor $\Meas \rightarrow \Ban_1$ into the category $\Ban_1$ of real Banach spaces and linear contractions. 
\end{PropSub}
\begin{proof}
Letting $(P,N)$ be a Hahn decomposition for $(Y,\SM f(\mu))$, we have that
$$(\SM f(\mu))^+(Y) = (\SM f(\mu))(P) = \mu(f^{-1}(P)) \lt \mu^+(f^{-1}(P)) \lt \mu^+(X)$$
and similarly $(\SM f(\mu))^-(Y) \lt \mu^-(X)$, from which the needed result follows.
\end{proof}

\section{Some basic lemmas on real-valued integration}

\begin{LemSub} \label{thm:basic_conv_lemma}
Let $f:X \rightarrow \RR$ be a measurable function.  Then there is a sequence $(\theta_i)$ of signed simple functions on $X$ with $|\theta_i| \lt |f|$ and $\theta_i \rightarrow f$ pointwise, such that for any finite signed measure $\mu$ on $X$, if $f$ is $\mu$-integrable then
$$\int f\;d\mu = \lim_i \int \theta_i \;d\mu\;.$$
\end{LemSub}
\begin{proof}
We may take some sequences $(\varphi_i)$, $(\psi_i)$ of nonnegative simple functions which converge pointwise from below to $f^+$ and $f^-$, respectively, and then it follows from the Monotone Convergence Theorem that $\theta_i := \varphi_i - \psi_i$ defines a sequence of simple functions with the needed properties, noting that $|\theta_i | \lt \varphi_i + \psi_i \lt f^+ + f^- = |f|$.
\end{proof}

\begin{PropSub} \label{thm:integ_of_composite}
Let $X \xrightarrow{f} Y \xrightarrow{g} \RR$ be measurable maps, and let $\mu \in \SM X$ be a signed measure on $X$.  If $g \circ f$ is $\mu$-integrable, then $g$ is $\SM f(\mu)$-integrable and
$$\int g \circ f \;d\mu = \int g \;d\,\SM f(\mu)\;.$$
\end{PropSub}
\begin{proof}
See 3.6.1 of \cite{Bo} and the remarks which follow there regarding signed measures.
\end{proof}

\section{Measures supported by a subset}

\begin{PropSub}
Let $i:A \rightarrow X$ be an initial measurable map.  Then the linear map $\SM i:\SM A \rightarrow \SM X$ is injective.
\end{PropSub}
\begin{proof}
Suppose $\SM i(\nu) = 0$.  By $\ref{thm:initial_final} (3)$, for each measurable $F \subs A$ there is some measurable $E \subs X$ with $F = i^{-1}(E)$, and $0 = (\SM i(\nu))(E) = \nu(i^{-1}(E)) = \nu(F)$.  Hence $\nu = 0$.
\end{proof}

\begin{DefSub} \label{def:support}
Let $X$ be a measurable space and $A \subs X$ an arbitrary subset.  We say that a measure $\mu \in \SM X$ \textit{is supported by $A$} if $\mu$ lies in the image of the injective linear map $\SM \iota_{AX}:\SM A \rightarrowtail \SM X$, where $\iota_{AX}:A \hookrightarrow X$ is the inclusion and, as usual, $A$ is endowed with the initial sigma-algebra induced by $\iota_{AX}$.  Hence $\mu$ is supported by $A$ iff $\mu$ is the direct image along $A \hookrightarrow X$ of some measure $\nu \in \SM A$.  As such a measure $\nu$ is necessarily unique if it exists, it is denoted by $\mu_A$ and called \textit{the restriction of $\mu$ to $A$}.  We define $\SM(A,X) := \{\mu \in \SM X \mid \text{$\mu$ is supported by $A$}\}$.
\end{DefSub}

\begin{RemSub}
For a \textit{measurable} subset $E \subs X$ we find that $E$ supports a measure $\mu \in \SM X$ iff $\mu(X \backslash E) = 0$, and in this case $\mu_E(F) = \mu(F)$ for all measurable $F \subs E$.
\end{RemSub}

\begin{RemSub}
Since $\SM(A,X)$ is the image of the injective linear map $\SM \iota_{AX}:\SM A \rightarrowtail \SM X$, $\SM(A,X)$ is a vector subspace of $\SM X$ isomorphic to $\SM A$.
\end{RemSub}

\begin{PropSub} \label{thm:integ_of_restn}
Let $f:X \rightarrow \RR$ be a measurable map, let $\mu \in \SM X$ be supported by a subset $A \subs X$, and suppose that the restriction $A \overset{\iota_{AX}}{\hookrightarrow} X \xrightarrow{f} \RR$ is $\mu_A$-integrable.  Then $f$ is $\mu$-integrable, and
$$\int f \;d\mu = \int f \circ \iota_{AX} d\mu_A\;.$$
\end{PropSub}
\begin{proof}
This follows from \ref{thm:integ_of_composite}.
\end{proof}

\begin{PropSub} \label{thm:total_var_of_restn}
Suppose that a signed measure $\mu \in \SM X$ is supported by a subset $A \subs X$.  Then $||\mu_A|| = ||\mu||$.
\end{PropSub}
\begin{proof}
Letting $(P',N')$ be a Hahn decomposition for $(A,\mu_A)$, we may take some measurable $P \subs X$ with $P' = A \cap P$.  Letting $N := X \backslash P$, one verifies readily that $(P,N)$ is a Hahn decomposition for $\mu$, and we have that $(P',N') = (A \cap P,A \cap N)$.  Using these Hahn decompositions, it is straightforward to verify that $\mu^+ = \SM \iota_{AX}((\mu_A)^+)$ and $\mu^- = \SM \iota_{AX}((\mu_A)^-)$, and the result follows.
\end{proof}

\begin{CorSub} \label{thm:isometry}
For each subset $A$ of a measurable space $X$, the injective linear map $\SM \iota_{AX}:\SM A \rightarrowtail \SM X$ restricts to an (isometric) isomorphism of normed vector spaces $\SM A \xrightarrow{\sim} \SM(A,X)$, whose inverse $\rho_{AX}:\SM(A,X) \rightarrow{} \SM A$ is given by $\mu \mapsto \mu_A$.
\end{CorSub}

\section{The endofunctor: Measures of bounded support}

\begin{DefSub}
Let $X \in \BornMeas$.  We say that a measure $\mu \in \SM X$ is \textit{of bounded support} if $\mu$ is supported by some bounded $B \subs X$.  We denote by $MX$ the set of all measures $\mu \in \SM X$ which are of bounded support.  For an arbitrary subset $A \subs X$, we denote by $M(A,X)$ the set of all $\mu \in MX$ which are supported by $A$.
\end{DefSub}

\begin{PropSub} \label{thm:morphisms_integrable}
For any morphism $f:X \rightarrow \RR$ in $\BornMeas$ and any $\mu \in MX$, $f$ is $\mu$-integrable.
\end{PropSub}
\begin{proof}
$\mu$ is supported by some bounded $B \subs X$, and the restriction $B \hookrightarrow X \xrightarrow{f} \RR$ has bounded image and hence is $\mu_B$-integrable, so the conclusion follows from \ref{thm:integ_of_restn}.
\end{proof}

\begin{RemSub}
Since we implicitly endow $MX$ with the initial sigma-algebra induced by the inclusion $MX \hookrightarrow \SM X$, it follows, with reference to Definition \ref{def:SX}, that $MX$ carries the initial sigma-algebra induced by the evaluation maps $\Ev_E:MX \rightarrow \RR$ associated to measurable subsets $E \subs X$.
\end{RemSub}

\begin{RemSub} \label{rem:MX_direct_limit}
For $X \in \BornMeas$, since the bornology $\B X$ is a directed poset (under the inclusion order), $MX$ is a directed union of the monotone increasing family $(\SM(B,X))_{B \in \B X}$ of vector subspaces of $\SM X$.  Hence $MX$ is a vector subspace of $\SM X$ isomorphic to the direct limit of the evident composite functor 
$$\B X \rightarrow \Meas \xrightarrow{\SM} \RVect\;.$$
\end{RemSub}

\begin{DefSub}
Let $X \in \BornMeas$.  For each bounded $B \subs X$ and each real number $\gamma \gt 0$, we let
$$M(B,X,\gamma) := \{\mu \in MX \mid \text{$\mu$ is supported by $B$ and $||\mu|| \lt \gamma$}\}\;.$$
\end{DefSub}

\begin{PropSub} \label{thm:basis_born_MX}
For $X \in \BornMeas$, the sets $M(B,X,\gamma)$ with $B \subs X$ bounded and $\gamma \gt 0$ constitute a basis for a bornology on $MX$.
\end{PropSub}
\begin{proof}
One verifies immediately, using the upward-directedness of $\B X$, that the given collection of sets is upward-directed.  Further, for any $\mu \in MX$, there is some bounded $B \subs X$ which supports $\mu$, so $\mu \in M(B,X,||\mu||)$.
\end{proof}

\begin{DefSub}
For $X \in \BornMeas$, we endow $MX$ with the \textit{supportwise bornology}, which is generated by the basis given in \ref{thm:basis_born_MX}.
\end{DefSub}

\begin{RemSub} \label{rem:norm_bornology}
For any bounded $B \in \BornMeas$, the supportwise bornology on $MB = \SM B$ coincides with the \textit{norm bornology}, which has a basis consisting of the closed balls about the origin.  For a general $X \in \BornMeas$, $MX$ is a direct limit $MX = \varinjlim_{B \in \B X}\SM B$ in $\RVect$ (\ref{rem:MX_direct_limit}), and by \ref{thm:isometry} one finds that $MX$ carries the direct limit bornology (\ref{rem:direct_limit_bornology}) induced by the norm bornologies on the spaces $\SM B$.  Each $\SM B$ is a \textit{bornological vector space} (see \ref{def:bvs_mvs_mbvs}), and hence by 2:8.2 of \cite{HN} we find that $MX$ is a direct limit, in the category of bornological vector spaces, of the normed vector spaces $\SM B$.
\end{RemSub}

\begin{PropSub}
Let $f:X \rightarrow Y$ be a $\BornMeas$-morphism.  Then the associated measurable linear map $\SM f:\SM X \rightarrow \SM Y$ restricts to a measurable linear map $Mf:MX \rightarrow MY$.  Moreover, $Mf$ is a bornological map, so we obtain a functor $M:\BornMeas \rightarrow \BornMeas$.
\end{PropSub}
\begin{proof}
For each bounded $B \subs X$, we have a measurable restriction $f_B:B \rightarrow f(B)$ of $f$, and by \ref{thm:N_yields_contractions}, the linear map
$$\SM f_{B}:\SM B \rightarrow \SM(f(B))$$
is bornological with respect to the norm bornologies.  In view of \ref{rem:norm_bornology}, these bornological linear maps induce a bornological linear map 
$$MX = \varinjlim_{B \in \B X}\SM B \longrightarrow \varinjlim_{C \in \B Y}\SM C = MY\;,$$
which is simply the desired restriction of $\SM f$.
\end{proof}

\section{The unit: Dirac measures}

\begin{DefSub} \label{def:unit}
Let $X$ be measurable space.  For each measurable $E \subs X$ we let $[E]:X \rightarrow \RR$ denote the characteristic function of $E$.  For each point $x \in X$, we denote by $\delta_x = \delta_{X,x} \in \SM X$ the \textit{Dirac measure} on $X$ associated to $x$, given by $\delta_x(E) = [E](x)$ for each measurable $E \subs X$.
\end{DefSub}

\begin{PropSub}
Let $X \in \BornMeas$ and let $x \in X$.  Then the Dirac measure $\delta_{X,x}$ is of bounded support.
\end{PropSub}
\begin{proof}
$\delta_{X,x}$ is supported by $\{x\}$, which is bounded.  Indeed, $\delta_{X,x}$ is the direct image along the inclusion $\{x\} \hookrightarrow X$ of the Dirac measure $\delta_{\{x\},x}$ on $\{x\}$ associated to $x$.
\end{proof}

\begin{DefSub}
For each $X \in \BornMeas$, we let $\delta_X:X \rightarrow MX$ be the map sending each $x \in X$ to the Dirac measure $\delta_X(x) = \delta_{X,x}$.
\end{DefSub}

\begin{PropSub}
The maps $\delta_X:X \rightarrow MX$, where $X \in \BornMeas$, are measurable and bornological and constitute a natural transformation $\delta:1_{\BornMeas} \rightarrow M$.
\end{PropSub}
\begin{proof}
Each such map $\delta_X$ is measurable, since for each measurable $E \subs X$, the composite
$$X \xrightarrow{\delta_X} MX \xrightarrow{\Ev_E} \RR$$
is equal to the characteristic function $[E]:X \rightarrow \RR$, which is measurable.  $\delta_X$ is also bornological, since for any bounded $B \subs X$ we find that $\delta_X(B) \subs M(B,X,1)$.  The naturality of $\delta$ is readily verified.
\end{proof}

\begin{PropSub} \label{thm:integ_wrt_delta_is_evaln}
Let $f:X \rightarrow \RR$ be a measurable function and let $x \in X$.  Then $f$ is $\delta_x$-integrable and
$$\int f \;d\delta_x = f(x)\;.$$
\end{PropSub}
\begin{proof}
This is a standard and easy exercise in applying the Monotone Convergence Theorem.  Establish the result in each of the following successive cases:  (i) when $f$ is a characteristic function, (ii) when $f$ is a simple function, (iii) when $f$ is nonnegative, and (iv) for general $f$.
\end{proof}

\section{The multiplication}
\begin{LemSub} \label{thm:EvE_bornmeas_morphism}
Let $X \in \BornMeas$.  Then 
\begin{enumerate}
\item For any bounded $B \subs X$, any $\gamma \gt 0$, and any measurable $E \subs X$, the image of $M(B,X,\gamma)$ under the evaluation map $\Ev_E:MX \rightarrow \RR$ is contained in $[-\gamma,\gamma]$.
\item For each measurable $E \subs X$, $\Ev_E:MX \rightarrow \RR$ is a $\BornMeas$-morphism.
\end{enumerate}
\end{LemSub}
\begin{proof}
(1) is readily verified, and (2) follows since $\Ev_E$ is measurable by the definition of $\sigma(MX)$.  
\end{proof}

\begin{DefSub}
Let $X \in \BornMeas$ and $\M \in MMX$.  By \ref{thm:EvE_bornmeas_morphism} and \ref{thm:morphisms_integrable}, we have that $\Ev_E:MX \rightarrow \RR$ is $\M$-integrable for each measurable $E \subs X$.  Hence we may define a real-valued set function $\kappa_X(\M)$ on $\sigma X$ by 
$$(\kappa_X(\M))(E) = \int \Ev_E \;d\M\;.$$
\end{DefSub}

\begin{PropSub}
Let $X \in \BornMeas$ and $\M \in MMX$.  Then $\kappa_X(\M)$ is a finite signed measure on $X$.
\end{PropSub}
\begin{proof}
Firstly, $(\kappa_X(\M))(\emptyset) = \int \Ev_\emptyset \;d\M = \int 0\;d\M = 0$.  Next, let $(E_i)_{i \in \NN}$ be a sequence of pairwise disjoint measurable subsets of $X$.  $\M$ is supported by some basic bounded subset $G = M(B,X,\gamma)$ of $MX$, where $B \subs X$ is bounded, and we let $\iota_G:G \hookrightarrow MX$ denote the inclusion.  Note that 
$$\left(\sum_{i = 1}^n \Ev_{E_i}\right)_{n \in \NN} \longrightarrow \;\;\Ev_{\cup_{i = 1}^{\infty}E_i}$$
pointwise on $MX$, by the countable additivity of the measures $\mu \in MX$.  Also, for each $n \in \NN$ we have that $\sum_{i = 1}^n \Ev_{E_i} = \Ev_{\cup_{i = 1}^nE_i}$ by finite additivity.  Moreover, for any measurable $E \subs X$ we have by Lemma \ref{thm:EvE_bornmeas_morphism} (1) that the restriction 
$$G \overset{\iota_G}{\hookrightarrow} MX \xrightarrow{\Ev_E} \RR$$
has $|\Ev_E \circ \;\iota_G| \lt \gamma$.  This applies in particular to the sets $\cup_{i = 1}^n E_i$ for each $n \in \NN$, so that
$$\left|\left(\sum_{i = 1}^n \Ev_{E_i}\right) \circ \iota_G\right| = \left|\Ev_{\cup_{i = 1}^nE_i} \circ \;\iota_G\right| \lt \gamma\;.$$
Hence we may employ the Bounded Convergence Theorem and Proposition \ref{thm:integ_of_restn} to compute as follows:
$$(\kappa_X(\M))\left(\bigcup_{i = 1}^\infty E_i\right) \;\;=\;\; \int \Ev_{\cup_{i = 1}^\infty E_i}\;d\M \;\;=\;\; \int \Ev_{\cup_{i = 1}^\infty E_i} \circ \;\iota_G\;d\M_G$$
$$=\;\; \lim_n \int \sum_{i = 1}^n \Ev_{E_i} \circ \;\iota_G\;d\M_G \;\;=\;\; \lim_n \sum_{i = 1}^n \int \Ev_{E_i} \circ \;\iota_G\;d\M_G$$
$$=\;\; \sum_{i = 1}^\infty \int \Ev_{E_i}\;d\M \;\;=\;\; \sum_{i = 1}^\infty (\kappa_X(\M))(E_i)\;.$$   
\end{proof}

\begin{LemSub} \label{thm:kappa_M_supported_by_B}
Let $X \in \BornMeas$, and suppose $\M \in MMX$ is supported by $M(B,X)$ for some bounded $B \subs X$.  Then $\kappa_X(\M)$ is supported by $B$.
\end{LemSub}
\begin{proof}
Consider the isomorphism of normed vector spaces $\rho_{BX}:M(B,X) \rightarrow MB$ of \ref{thm:isometry}, given by $\mu \mapsto \mu_B$.  By \ref{rem:norm_bornology}, $\rho_{BX}$ is an isomorphism of bornological sets.  Further, $\rho_{BX}$ is measurable, and hence a $\BornMeas$-morphism, since for each measurable $F \subs B$, which must be of the form $F = B \cap E$ for some measurable $E \subs X$, one finds that the diagram
\begin{equation} \label{eq:rho_diagram}
\xymatrix {
{M(B,X)} \ar[d]_j \ar[r]^{\rho_{BX}} & MB \ar[d]^{\Ev_{B \cap E}}\\
MX \ar[r]^{\Ev_E}                    & \RR 
}
\end{equation}
commutes, where $j$ is the inclusion, so that the composite $\Ev_F \circ \rho_{BX} = \Ev_E \circ j$ is measurable.

For each measurable $E \subs X$, we again employ the commutativity of the diagram \eqref{eq:rho_diagram} to compute as follows:
\begin{align*}
(\kappa_X(\M))(E) &= \int \Ev_E\;d\M                                                    &\\
                  &= \int \Ev_E\: \circ \;\:j\;\;d\M_{M(B,X)}                                   &\text{(by \ref{thm:integ_of_restn})}\\
                  &= \int \Ev_{B \cap E}\: \circ \;\rho_{BX}\;\:d\M_{M(B,X)}                  &\text{(by \eqref{eq:rho_diagram})}\\
                  &= \int \Ev_{B \cap E}\;d\,M\rho_{BX}(\M_{M(B,X)})                    &\text{(by \ref{thm:integ_of_composite})}\\
                  &= (\kappa_B \circ M\rho_{BX}(\M_{M(B,X)}))(B \cap E)                 &\\
                  &= (\SM \iota_{BX} \circ \kappa_B \circ M\rho_{BX}(\M_{M(B,X)}))(E)\;.&
\end{align*}
Hence $\kappa_X(\M)$ is the direct image along $\iota_{BX}:B \hookrightarrow X$ of the measure 
$$\kappa_B \circ M\rho_{BX}(\M_{M(B,X)})$$
on $B$.
\end{proof}

\begin{CorSub}
Let $X \in \BornMeas$ and $\M \in MMX$.  Then $\kappa_X(\M) \in MX$.
\end{CorSub}

\begin{PropSub} \label{thm:lifts_of_real_fns}
Let $f:X \rightarrow \RR$ be a morphism in $\BornMeas$.  Then 
\begin{enumerate}
\item there is a $\BornMeas$-morphism $f^\sharp:MX \rightarrow \RR$ given by $f^\sharp(\mu) = \int f\;d\mu$, and
\item for any $\M \in MMX$,
$$\int f^\sharp \;d\M = \int f\;d\,\kappa_X(\M)\;.$$
\end{enumerate}
\end{PropSub}
\begin{proof}
\textbf{(i).}  First consider the case where $f = [E]$ is the characteristic function of some measurable $[E] \subs X$.  We have that $f^\sharp(\mu) = \int f \;d\mu = \int [E]\;d\mu = \mu(E)$ for each $\mu \in MX$, so $f^\sharp = \Ev_E : MX \rightarrow \RR$ is a $\BornMeas$-morphism, using Lemma \ref{thm:EvE_bornmeas_morphism}.  Further,
$$\int f^\sharp \,d\M = \int \Ev_E\,d\M = (\kappa_X(\M))(E) = \int [E]\;d\,\kappa_X(\M) = \int f\;d\,\kappa_X(\M)\;.$$

\paragraph*{(ii).}  It follows from the linearity of the integral that the map $(-)^\sharp:\BornMeas(X,\RR) \rightarrow \Set(MX,\RR)$ is linear.  Hence for any signed simple function $f = \sum_{i = 1}^n a_i[E_i]$ on $X$,  $f^\sharp$ is a linear combination of the $\BornMeas$-morphisms $[E_i]^\sharp : MX \rightarrow \RR$ and hence is a $\BornMeas$ morphism, by \ref{thm:BornMeas_functions_vector_space}.  The needed equation (2) for $f$ follows from $(i)$ and the linearity of the integral.

\paragraph*{(iii).}  For general $f$, we have by \ref{thm:basic_conv_lemma} that there is a sequence $(\theta_i)$ of signed simple functions on $X$ such that $|\theta_i| \lt |f|$, $\theta_i \rightarrow f$ pointwise, and
$$\forall \mu \in MX \;:\; f^\sharp(\mu) = \int f\;d\mu = \lim_i \int \theta_i \;d\mu = \lim_i \theta_i^\sharp(\mu).$$
Hence $f^\sharp = \lim_i \theta_i^\sharp$ pointwise on $MX$, so since each $\theta_i^\sharp$ is measurable by (ii), $f^\sharp$ is measurable.  

The following general observation will enable the remainder of our proof:  \\
\\
\noindent \textbf{Claim.}
For any basic bounded subset $M(B,X,\gamma)$ of $MX$, any $\beta \gt 0$, and any bornological measurable function $g:X \rightarrow \RR$ with $|g| \lt \beta$ on $B$,
$$|g^\sharp| \lt \beta\gamma\;\;\;\; \text{on $M(B,X,\gamma)$}\;.$$
\\
Indeed, for each $\mu \in M(B,X,\gamma)$ we find, using \ref{thm:integ_of_restn} and \ref{thm:total_var_of_restn}, that
$$|g^\sharp(\mu)| = \left|\int g \;d\mu \right| = \left|\int g \circ \iota_{BX}\;d\mu_B \right| \lt \int |g \circ \iota_{BX}|\;d|\mu_B|$$
$$\lt \int \beta\;d|\mu_B| = \beta |\mu_B|(B) = \beta||\mu_B|| = \beta||\mu|| \lt \beta\gamma\;.$$

The first consequence of this Claim is that $f^\sharp$ is bornological, since for any basic bounded subset $M(B,X,\gamma)$ of $MX$ we can take $\beta \gt 0$ with $|f| \lt \beta$ on $B$, as $f$ is bornological, so the Claim applies.

Secondly, the Claim allows a proof of the equation in (2), as follows.  Each $\M \in MMX$ is supported by some basic bounded subset $G = M(B,X,\gamma)$ of $MX$, and, taking any bound $\beta$ for $|f|$ on $B$ we have that $|\theta_i| \lt |f| \lt \beta$ on $B$ for every $i \in \NN$.  Hence, by the Claim,
$$|\theta_i \circ \iota_B| \lt \beta\;\;\text{and}\;\; |\theta_i^\sharp \circ \iota_G| \lt \beta\gamma\;,\;\;\;\;\text{ for each $i \in \NN$,}$$
where $\iota_B:B \hookrightarrow X$ and $\iota_G:G \hookrightarrow MX$ are the inclusions.  We also have that
$$f \circ \iota_B = \lim_i\;\theta_i \circ \iota_B \;\;\;\;\text{and}\;\;\;\; f^\sharp \circ \iota_G = \lim_i\;\theta_i^\sharp \circ \iota_G$$
pointwise.  Hence, we may apply the Bounded Convergence Theorem twice in order to compute that
\begin{align*}
\int f^\sharp \;d\M &= \int f^\sharp \circ \iota_G \;d\M_G                      & \text{(by \ref{thm:integ_of_restn})} \\
\                   &= \lim_i \int \theta_i^\sharp \circ \iota_G \;d\M_G        & \text{(by the B.C.T.)}               \\
                    &= \lim_i \int \theta_i^\sharp \;d\M                        & \text{(by \ref{thm:integ_of_restn})} \\
                    &= \lim_i \int \theta_i \;d\,\kappa_X(\M)                   & \text{(by (ii))}                     \\
                    &= \lim_i \int \theta_i \circ \iota_B \;\;d\,\kappa_X(\M)_B & \text{(by \ref{thm:integ_of_restn})} \\
                    &= \int f \circ \iota_B \;\;d\,\kappa_X(\M)_B               & \text{(by the B.C.T)}                \\
                    &= \int f \;d\,\kappa_X(\M)\;,                              & \text{(by \ref{thm:integ_of_restn})}  \\
\end{align*}
 since $\M$ is supported by $G$ and, by \ref{thm:kappa_M_supported_by_B}, $\kappa_X(\M)$ is supported by $B$.
\end{proof}

\begin{PropSub}
Let $X \in \BornMeas$.  Then the map $\kappa_X:MMX \rightarrow MX$ is a $\BornMeas$-morphism.
\end{PropSub}
\begin{proof}
Firstly, $\kappa_X$ is measurable, since for each measurable $E \subs X$, one checks that the composite $MMX \xrightarrow{\kappa_X} MX \xrightarrow{\Ev_E} \RR$ is none other than $\Ev_E^\sharp$, which is measurable by \ref{thm:lifts_of_real_fns}.

Secondly, $\kappa_X$ is bornological, as follows.  The bornology on $MMX$ has a basis consisting of the sets $M(G,MX,\delta)$, where $G = M(B,X,\gamma)$ is a basic bounded subset of $MX$.  For any such, we shall show that
$$\kappa_X(M(G,MX,\delta)) \subs M(B,X,\gamma\delta)\;,$$
yielding the needed result.  To this end, let $\M \in M(G,MX,\delta)$.  Then $\M$ is supported by $G = M(B,X,\gamma)$, so by Lemma \ref{thm:kappa_M_supported_by_B}, $\kappa_X(\M) \in M(B,X)$ and hence it suffices to show that $||\kappa_X(\M)|| \lt \gamma\delta$.  Let $(P,N)$ be a Hahn decomposition for $(X,\kappa_X(\M))$.  Notice that $t := \Ev_P - \Ev_N : MX \rightarrow \RR$ is the function sending each $\mu \in MX$ to its total variation
$$t(\mu) = \mu(P) - \mu(N) = \mu^+(X) + \mu^-(X) = ||\mu||\;.$$
Moreover,
$$||\kappa_X(\M)|| = (\kappa_X(\M))(P) - (\kappa_X(\M))(N)$$
$$= \int \Ev_P \;d\M - \int \Ev_N \;d\M = \int \Ev_P - \Ev_N \;d\M\;,$$
so
$$||\kappa_X(\M)|| = t(\kappa_X(\M)) = \int t \;d\M\;.$$
For each $\mu \in G = M(B,X,\gamma)$ we have $t(\mu) = ||\mu|| \lt \gamma$, so
$$|t \circ \iota_G| \lt \gamma\;,$$
where $\iota_G:G \hookrightarrow MX$ is the inclusion.  Hence, since $\M$ is supported by $G$,
$$||\kappa_X(\M)|| = \int t \;d\M = \int t \circ \iota_G \;d\M_G \lt \int |t \circ \iota_G|\;d|\M_G|$$
$$\lt \gamma|\M_G|(G) = \gamma ||\M_G|| = \gamma ||\M|| \lt \gamma\delta\;,$$
using Proposition \ref{thm:total_var_of_restn} and the assumption that $\M \in M(G,MX,\delta)$.
\end{proof}

\begin{PropSub}
The $\BornMeas$-morphisms $\kappa_X:MMX \rightarrow MX$ constitute a natural transformation $\kappa:MM \rightarrow M$.
\end{PropSub}
\begin{proof}
Let $f:X \rightarrow Y$ in $\BornMeas$.  For each $\M \in MMX$ and each measurable $F \subs Y$, since the composite $MX \xrightarrow{Mf} MY \xrightarrow{\Ev_F} \RR$
is equal to the evaluation map $\Ev_{f^{-1}(F)}$, we compute, using Proposition \ref{thm:integ_of_composite}, that
$$
\begin{array}{lllllll}
(\kappa_Y \circ MMf(\M))(F) & = & \int \Ev_F\;d\,MMf(\M)      & = & \int \Ev_F \circ Mf\;d\M     \\
                            & = & \int \Ev_{f^{-1}(F)}\;d\M   & = & (\kappa_X(\M))(f^{-1}(F)) \\
                            & = & (Mf \circ \kappa_X (\M))(F) &   & 
\end{array}
$$
\end{proof}

\section{The monad of finite signed measures of bounded support}

\begin{ThmSub}
$\MM := (M,\delta,\kappa)$ is a monad on $\BornMeas$.
\end{ThmSub}
\begin{proof}
It remains only to establish the unit and associativity laws
$$\kappa \cdot \delta M = 1_M = \kappa \cdot M \delta \;\;\;\;\text{and}\;\;\;\;\kappa \cdot M\kappa = \kappa \cdot \kappa M\;\;.$$

For each $\mu \in MX$ and each measurable $E \subs X$, we deduce that
$$(\kappa_X \circ \delta_{MX}(\mu))(E) = \int \Ev_E \;d\delta_\mu = \Ev_E(\mu) = \mu(E)$$
by \ref{thm:integ_wrt_delta_is_evaln}.  Also, using Proposition \ref{thm:integ_of_composite}
$$
\begin{array}{lllllll}
(\kappa_X \circ M\delta_X (\mu))(E) & = & \int \Ev_E\;d\,M\delta_X(\mu) & = & \int \Ev_E \circ \delta_X \;d\mu \\
                                    & = & \int [E]\;d\mu                & = & \mu(E)\;.
\end{array}
$$

For the associativity law, let $\mathfrak{M} \in MMMX$.  For each measurable $E \subs X$, since the composite $MMX \xrightarrow{\kappa_X} MX \xrightarrow{\Ev_E} \RR$ is $\Ev_E^\sharp$ (see \ref{thm:lifts_of_real_fns}), we compute, using Propositions \ref{thm:integ_of_composite} and \ref{thm:lifts_of_real_fns} that
$$
\begin{array}{lllllll}
 (\kappa_X \circ M\kappa_X (\mathfrak{M}))(E) & = & \int \Ev_E \;d\,M\kappa_X(\mathfrak{M})          & = & \int \Ev_E \circ \kappa_X \;d\mathfrak{M} \\
                                              & = & \int \Ev_E^\sharp \;d\mathfrak{M}                               & = & \int \Ev_E \;d\,\kappa_{MX}(\mathfrak{M}) \\
                                              & = & (\kappa_X \circ \kappa_{MX}(\mathfrak{M}))(E)\;. &   &
\end{array}
$$
\end{proof}

\section{The vector space structure on $\MM$-algebras}

\begin{DefSub}
Let $\LL = (L,\varsigma,\tau)$ be the monad induced by the adjunction between the forgetful functor $\RVect \rightarrow \Set$ and its left adjoint.  Hence $L:\Set \rightarrow \Set$ associates to each set $X$ the (set underlying the) free vector space
$$LX = \bigoplus_{x \in X}\RR x$$
generated by $X$, consisting of formal linear combinations of the elements of $X$.  The map $\varsigma_X:X \rightarrow LX$ is just the injection of generators and may be taken to be a subset inclusion.
\end{DefSub}

\begin{RemSub}
Recall that $\RVect$ is isomorphic to the category of algebras $\Set^\LL$ of $\LL$.
\end{RemSub}

\begin{DefSub} \label{def:Delta}
Let $U:\BornMeas \rightarrow \Set$ be the forgetful functor.  For each object $X \in \BornMeas$, since the underlying set $UMX$ of $MX$ carries the structure of a real vector space, the function $U\delta_X:UX \rightarrow UMX$ (\ref{def:unit}) induces a unique linear map $\Delta_X:LUX \rightarrow UMX$ such that
$$
\xymatrix{
{UX} \ar@{^{(}->}[r]^{\varsigma_{UX}} \ar[dr]_{U\delta_X} & {LUX} \ar[d]^{\Delta_X} \\
                                                       & {UMX}                   
}
$$
commutes.
\end{DefSub}

\begin{LemSub}
The maps $\Delta_X$ of Definition \ref{def:Delta} constitute a natural transformation $\Delta:LU \rightarrow UM$.
\end{LemSub}
\begin{proof}
Let $f:X \rightarrow Y$ in $\BornMeas$.  All the morphisms in the diagram
$$
\xymatrix{
{LUX} \ar[r]^{\Delta_X} \ar[d]_{LUf} & {UMX} \ar[d]^{UMf} \\
{LUY} \ar[r]^{\Delta_Y}              & {UMY}              \\
}
$$
are linear maps with respect to the given vector space structures.  Hence it suffices to check the commutativity of this diagram on each element $x$ of the basis $UX$ for $LUX$, and indeed
$$UMf \circ \Delta_X(x) = Mf(\delta_X(x)) = \delta_Y(f(x)) = \Delta_Y(f(x)) = \Delta_Y \circ LUf(y)$$
by the naturality of $\delta:1_{\BornMeas} \rightarrow M$.
\end{proof}

\begin{PropSub}
The forgetful functor $U:\BornMeas \rightarrow \Set$ and the natural transformation $\Delta:LU \rightarrow UM$ constitute a \emph{monad morphism (\cite{St}\cite{LS})}
$$\mathbf{\Delta} := (U,\Delta):\MM \rightarrow \LL\;.$$
\end{PropSub}
\begin{proof}
By its very definition, $\Delta$ satisfies the equation $\Delta \cdot \varsigma U = U\delta$.  Hence it suffices to show that the diagram
$$
\xymatrix{
{LLU} \ar[r]^{L\Delta} \ar[d]_{\tau U} & {LUM} \ar[r]^{\Delta M} & {UMM} \ar[d]^{U\kappa} \\
{LU}  \ar[rr]^{\Delta}                  &                         & {UM}
}
$$
commutes.  It is clear from the definition of $\kappa$ that its components $\kappa_X:MMX \rightarrow MX$ are linear maps.  In fact, each component of each of the natural transformations in the given diagram is linear with respect to the given vector space structures.  Hence it suffices to show that the composites
$$LU \xrightarrow{\varsigma LU} LLU \xrightarrow{L\Delta} LUM \xrightarrow{\Delta M} UMM \xrightarrow{U\kappa} UM$$
$$LU \xrightarrow{\varsigma LU} LLU \xrightarrow{\tau U} LU \xrightarrow{\Delta} UM$$
are equal, and indeed we compute that
$U\kappa \cdot \Delta M \cdot L\Delta \cdot \varsigma LU = U\kappa \cdot \Delta M \cdot \varsigma UM \cdot \Delta = U\kappa \cdot U\delta M \cdot \Delta = \Delta = \Delta \cdot \tau U \cdot \varsigma LU$, using the naturality of $\varsigma$, the equation $\Delta \cdot \varsigma U = U\delta$, and the unit laws for $\MM$ and $\LL$.
\end{proof}

\begin{CorSub} \label{thm:Malg_vect_sp}
The monad morphism $\mathbf{\Delta} = (U,\Delta):\MM \rightarrow \LL$ induces a functor
$$U^\mathbf{\Delta}:\BornMeas^{\MM} \rightarrow \Set^{\LL} \cong \RVect$$
which endows every $\MM$-algebra with the structure of a real vector space, and every $\MM$-homomorphism is thus a linear map.

For an $\MM$-algebra $(X,c:MX \rightarrow X)$, the addition and scalar multiplication maps of the associated vector space are the composites
$$X \times X \xrightarrow{\delta_X \times \delta_X} MX \times MX \xrightarrow{+} MX \xrightarrow{c} X$$
$$\RR \times X \xrightarrow{1_\RR \times \delta_X} \RR \times MX \xrightarrow{\cdot} MX \xrightarrow{c} X\;.$$ 
The vector space structure associated to the free $\MM$-algebra $MX$ by $U^\mathbf{\Delta}$ coincides with the given structure on $MX$ (\ref{rem:MX_direct_limit}).
\end{CorSub}
\begin{proof}
Any monad morphism induces such a functor (see \cite{LS}, \S 2.1), which in the present case is given on objects by $U^\mathbf{\Delta}(X,c) = (UX,LUX \xrightarrow{\Delta_X} UMX \xrightarrow{Uc} UX)$ and commutes with the forgetful functors to $\Set$.  The addition and scalar multiplication maps of the associated vector space are the composites
$$UX \times UX \xrightarrow{\varsigma_{UX} \times \varsigma_{UX}} LUX \times LUX \xrightarrow{+} LUX \xrightarrow{\Delta_X} UMX \xrightarrow{Uc} UX$$
$$\RR \times UX \xrightarrow{1_\RR \times \varsigma_{UX}} \RR \times LUX \xrightarrow{\cdot} LUX \xrightarrow{\Delta_X} UMX \xrightarrow{Uc} UX \;.$$
The first of these coincides with the first composite given above, since the diagram
$$
\xymatrix{
UX \times UX \ar[rr]^{\varsigma_{UX} \times \varsigma_{UX}} \ar[drr]_{U \delta_X \times U \delta_X} & & LUX \times LUX \ar[d]^{\Delta_X \times \Delta_X} \ar[r]^(0.6){+} & LUX \ar[d]^{\Delta_X}    &    \\
                                                                                                  & & UMX \times UMX \ar[r]^(0.6){+}                                   & UMX \ar[r]^{Uc}          & UX \\
}
$$
commutes, as $\mathbf{\Delta}$ is a monad morphism and $\Delta_X$ is a linear map.  We reason analogously with regard to the second composite.

The addition operation with which the functor $U^\mathbf{\Delta}$ endows a free $\MM$-algebra $MX$ coincides with the usual addition operation on measures, since the diagram
$$
\xymatrix{
MX \times MX \ar[rr]^{\delta_{MX} \times \delta_{MX}} \ar@{=}[drr] & & MMX \times MMX \ar[d]^{\kappa_X \times \kappa_X} \ar[r]^(0.6){+} & MMX \ar[d]^{\kappa_X} \\
                                                                   & & MX \times MX \ar[r]^(0.6){+}                                     & MX                    \\
}
$$
commutes, using a unit law for $\MM$ and the fact that $\kappa_X$ is linear with respect to the usual operations on $MX$.  Analogous reasoning applies with regard to the scalar multiplication operation.
\end{proof}

\begin{CorSub} \label{thm:struct_map_linear}
For any $\MM$-algebra $(X,c)$, the structure map $c:MX \rightarrow X$ is linear. 
\end{CorSub}
\begin{proof}
This follows from \ref{thm:Malg_vect_sp}, since $c$ is an $\MM$-homomorphism.
\end{proof}

\section{$\MM$-algebras as bornological vector spaces}
\begin{LemSub} \label{thm:R_ring_obj_in_BornMeas}
$\RR$ is a ring object in $\BornMeas$.  Equivalently, the addition and multiplication operations $+,\cdot:\RR \times \RR \rightarrow \RR$ are measurable and bornological, where the product $\RR \times \RR$ is taken in $\BornMeas$. 
\end{LemSub}
\begin{proof}
It is straightforward to show that the given maps are bornological.  Also, since these maps are continuous, they are Borel measurable, and the Borel $\sigma$-algebra on $\RR \times \RR$ coincides with the product sigma-algebra (e.g., by \cite{Bo}, 6.4.2).
\end{proof}

\begin{DefSub} \label{def:bvs_mvs_mbvs}
A (real) \textit{bornological vector space} (resp. \textit{measurable vector space}, \textit{measurable bornological vector space}) is an $\RR$-vector-space object in $\Born$ (resp. $\Meas$, $\BornMeas$).  Hence a real vector space $V$ is a bornological (resp. measurable, measurable bornological) vector space if $V$ is endowed with a bornology (and/or sigma-algebra) making the addition and scalar multiplication maps bornological (and/or measurable) as maps defined on the products $V \times V$, $\RR \times V$ taken in $\Born$ (resp. $\Meas$, $\BornMeas$).  We define associated categories $\RVect(\Born)$, $\RVect(\Meas)$, and $\RVect(\BornMeas)$, whose morphisms are linear maps that are, accordingly, bornological and/or measurable.  It is conventional to use the term \textit{bounded linear map} to mean bornological linear map.
\end{DefSub}

\begin{LemSub} \label{thm:MX_mbvs}
Let $X \in \BornMeas$.  Then $MX$ is a measurable bornological vector space.  Hence we obtain a functor $M:\BornMeas \rightarrow \RVect(\BornMeas)$.
\end{LemSub}
\begin{proof}
For each measurable $E \subs X$, the diagrams
$$
\xymatrix{
MX \times MX \ar[r]^{+} \ar[d]_{\Ev_E \times \Ev_E} & MX \ar[d]^{\Ev_E} & & \RR \times MX  \ar[r]^{\cdot} \ar[d]_{1_\RR \times \Ev_E} & MX \ar[d]^{\Ev_E} \\
\RR \times \RR \ar[r]^{+}                           & \RR               & & \RR \times \RR \ar[r]^{\cdot}                             & \RR
}
$$
commute since the evaluation map $\Ev_E$ is linear, so since the bottom-left composites are measurable (using \ref{thm:R_ring_obj_in_BornMeas}), the top-right are as well.  Hence the addition and scalar multiplication maps of $MX$ are measurable.  Also, we know from \ref{rem:norm_bornology} that $MX$ is a bornological vector space.
\end{proof}

\begin{PropSub} \label{thm:Malg_mbvs}
Let $(X,c)$ be an $\MM$-algebra.  Then $X$, endowed the associated vector space structure (\ref{thm:Malg_vect_sp}), is a measurable bornological vector space.  Hence the functor $\BornMeas^{\MM} \rightarrow \RVect$ of \ref{thm:Malg_vect_sp} factors through $\RVect(\BornMeas)$.
\end{PropSub}
\begin{proof}
Corollary \ref{thm:Malg_vect_sp} exhibits the addition and scalar multiplication maps of $X$ as composites of what are, by Lemma \ref{thm:MX_mbvs}, measurable bornological maps.
\end{proof}

\begin{DefSub}
A bornological vector space $V$ is \textit{convex} \cite{HN} if the bornology $\B V$ on $V$ has a basis of convex sets; equivalently, for each bounded $B \subs V$, there is some convex bounded subset $C \subs V$ with $B \subs C$.  We define $\RConvBvs$ to be the full subcategory of $\RVect(\Born)$ consisting of convex bornological vector spaces.
\end{DefSub}

\begin{RemSub}
Every convex bornological vector space $V$ acquires the structure of a locally convex topological vector space when we take as a neighbourhood basis for the origin $0 \in V$ the \textit{bornivorous discs}; see \cite{HN}.  This passage is part of an adjunction between $\RConvBvs$ and the category of locally convex spaces; see \cite{FK}.
\end{RemSub}

\begin{ThmSub}
Let $(X,c)$ be an $\MM$-algebra.  Then $X$, endowed the associated vector space structure (\ref{thm:Malg_vect_sp}), is a convex bornological vector space.

Hence the functor $\BornMeas^{\MM} \rightarrow \RVect$ of \ref{thm:Malg_vect_sp} factors through \\$\RConvBvs$.
\end{ThmSub}
\begin{proof}
By \ref{thm:Malg_mbvs}, it suffices to show that the bornology $\B X$ on $X$ has a basis of convex sets.  Consider any bounded $B \subs X$.   Let $P(B,X)$ be the set of all probability measures on $X$ supported by $B$; i.e., $P(B,X) = PX \cap M(B,X)$ where $PX := \{\mu \in \SM X \mid \mu \gt 0,\;||\mu|| = 1\}$ is the set of all probability measures on $X$.  Then, since $PX$ is a convex subset of the space $\SM X$ of finite signed measures and $M(B,X)$ is a vector subspace of $MX$, $P(B,X)$ is a convex subset of $MX$.  Hence, since $P(B,X) \subs M(B,X,1)$, $P(B,X)$ is, moreover, a bounded convex subset of $MX$.  By \ref{thm:struct_map_linear}, $c:MX \rightarrow X$ is a bornological linear map, so the image $c(P(B,X))$ of $P(B,X)$ under $c$ is a bounded convex subset of $X$.  Further, $B \subs c(P(B,X))$, since for each $x \in B$ we have that the Dirac measure $\delta_x$ is a probability measure supported by $B$, i.e. $\delta_x \in P(B,X)$, and $c(\delta_x) = c \circ \delta_X(x) = x$. 
\end{proof}

\section{Integrals valued in an $\MM$-algebra}

\begin{PropSub}
$\RR$ is an $\MM$-algebra with structure map $c_\RR:M\RR \rightarrow \RR$ given by
$$c_\RR(\mu) = \int \id_\RR \;d\mu\;,$$
where the right-hand-side is the Lebesgue integral of the identity map $\id_\RR:\RR \rightarrow \RR$ with respect to $\mu$.  In fact, $\RR$ is isomorphic to the free $\MM$-algebra $M1$ on the one-point measurable bornological set $1 = \{*\}$.
\end{PropSub}
\begin{proof}
One readily checks that the map $\Ev_1:M1 \rightarrow \RR$ is an isomorphism in $\BornMeas$, with inverse given by $\alpha \mapsto \alpha\delta_*$.  Hence it suffices to show that
$$
\xymatrix{
MM1 \ar[r]^{M\Ev_1} \ar[d]^{\kappa_1} & M\RR \ar[d]^{c_\RR}\\
M1 \ar[r]^{\Ev_1}                     & \RR      
}
$$
commutes, and indeed, for each $\M \in MM1$ we have by \ref{thm:integ_of_composite} that $\Ev_1 \circ \kappa_1(\M) = \int \Ev_1 \;d\M = \int \id_\RR \;d\,M\Ev_1(\M) = c_\RR \circ M\Ev_1(\M)$.
\end{proof}

\begin{RemSub}
For any $\BornMeas$-morphism $f:T \rightarrow \RR$ and any $\mu \in MT$, the integral of $f$ with respect to $\mu$ may be expressed in terms of the structure map $c_\RR$ on the $\MM$-algebra $\RR$ as
$$\int f\;d\mu = \int \id_{\RR}\;d\,Mf(\mu) = c_\RR \circ Mf(\mu)\;,$$
using Proposition \ref{thm:integ_of_composite}.  This motivates the following definition:
\end{RemSub}

\begin{DefSub} \label{def:integ}
For an $\MM$-algebra $(X,c)$, a $\BornMeas$-morphism $f:T \rightarrow X$, and any $\mu \in MT$, we define \textit{the integral of $f$ with respect to $\mu$} to be
$$\int f \;d \mu := \int_{t \in T}f(t)\;d\mu := c \circ Mf(\mu)\;.$$
\end{DefSub}

\begin{RemSub}
Let $(X,c)$ be an $\MM$-algebra.  Then, regarding $(X,c)$ as an algebra of the algebraic theory (over $\BornMeas$) associated to $\MM$, we have for each $T \in \BornMeas$ and each $\mu \in MT$ an \textit{operation} $\Omega_\mu^T$ on $X$ of \textit{arity} $T$ associated to $\mu$, namely the function
$$\Omega_\mu^T:\BornMeas(T,X) \rightarrow X\;,\;\;\;\;f \mapsto c \circ Mf(\mu)\;,$$
and in view of Definition \ref{def:integ}, this is exactly the operation of $X$-valued \textit{integration} with respect to $\mu$, given by 
$$\Omega_\mu^T(f) = \int f \;d\mu\;.$$
\end{RemSub}

\begin{PropSub} \label{thm:homom}
Let $X$ and $Y$ be $\MM$-algebras and $\phi:X \rightarrow Y$ a $\BornMeas$-morphism.  Then $\phi$ is an $\MM$-homomorphism iff
$$\phi\left(\int f\;d\mu\right) = \int \phi \circ f \;d\mu\;,\;\;\;\;\;\;\text{i.e.},\;\;\phi\left(\int_{t \in T} f(t)\;d\mu\right) = \int_{t \in T} \phi(f(t)) \;d\mu\;,$$
for all $f:T \rightarrow X$ in $\BornMeas$ and $\mu \in MT$.
\end{PropSub}
\begin{proof}
The verification is straightforward.
\end{proof}

\begin{ExaSub} \label{exa:EvE_homom}
It is immediate from the definition of the structure map $\kappa_X$ of a free $\MM$-algebra $MX$ that the evaluation maps $\Ev_E:MX \rightarrow \RR$ are $\MM$-homomorphisms.
\end{ExaSub}

\begin{ExaSub}
Since $\BornMeas$ is complete (\ref{thm:bornmeas_topological}), $\BornMeas^{\MM}$ is complete.  Hence for any set $n$ there is a product $\RR^n$ in the category of $\MM$-algebras, and the underlying $\BornMeas$-object $\RR^n$ is simply the product in $\BornMeas$.  Since the projections $\pi_i:\RR^n \rightarrow \RR$ ($i \in n$) are $\MM$-homomorphisms, $\RR^n$ carries the \textit{coordinatewise integral}, given by
$\pi_i\left(\int f \;d\mu\right) = \int \pi_i \circ f \;d\mu$
for all $f:T \rightarrow \RR^n$ in $\BornMeas$, $\mu \in MT$, and $i \in n$.
\end{ExaSub}

\begin{RemSub} \label{rem:sharp}
For an $\MM$-algebra $(X,c)$ and any $f:T \rightarrow X$ in $\BornMeas$, the composite $MT \xrightarrow{Mf} MX \xrightarrow{c} X$, $\mu \mapsto \int f \;d\mu$, is the $\MM$-homomorphism induced by $f$, which we denote by $f^\sharp$ and call the \textit{lift} of $f$.  We refer the reader to \cite{MW}, Theorem 8.2, for some important properties of the lift combinator $(-)^\sharp$ for a general monad.
\end{RemSub}

\begin{LemSub}
For an $\MM$-algebra $(X,c)$ and an object $T \in \BornMeas$, $\BornMeas(T,X)$ is a real vector space under the pointwise operations.
\end{LemSub}
\begin{proof}
For $f,g \in \BornMeas(T,X)$ and $a \in \RR$, the pointwise sum $f+g$ and scalar multiple $af$ are the composites
$$T \xrightarrow{(f,g)} X \times X \xrightarrow{+} X \;\;\;\text{and}\;\;\;T \xrightarrow{(a,f)} \RR \times X \xrightarrow{\cdot} X\;,$$
which are $\BornMeas$-morphisms since the the addition and scalar multiplication maps $+,\cdot$ are $\BornMeas$-morphisms, by \ref{thm:Malg_mbvs}.
\end{proof}

\begin{ThmSub}
For an $\MM$-algebra $(X,c)$, an object $T \in \BornMeas$, and any $\mu \in MT$, the operation
$$\Omega_\mu^T:\BornMeas(T,X) \rightarrow X\;,\;\;\;\;f \mapsto \int f\;d\mu = c \circ Mf(\mu)$$
is a linear map.
\end{ThmSub}
\begin{proof}
The needed linearity of integration is equivalent to the requirement that for all $f,g \in \BornMeas(T,X)$ and all $a,b \in \RR$, $c \circ M(af + bg) = ac \circ Mf + bc \circ Mg$, i.e. $(af+bg)^\sharp = af^\sharp + bg^\sharp$.

We know that the $\MM$-algebra $\RR$ has this property.  It follows that the free algebra $MX$ does as well, as follows.  Let $h,k \in \BornMeas(T,MX)$, $a,b \in \RR$, and $E \subs X$ measurable.  Since $\Ev_E$ is an $\MM$-homomorphism (\ref{exa:EvE_homom}) and a linear map, we find, using properties of the lift combinator (\ref{rem:sharp}), that 
$$\Ev_E \circ (ah + bk)^\sharp = (\Ev_E \circ (ah + bk))^\sharp = (a\Ev_E \circ h + b\Ev_E \circ k)^\sharp$$
$$= a(\Ev_E \circ h)^\sharp + b(\Ev_E \circ k)^\sharp = a\Ev_E \circ h^\sharp + b\Ev_E \circ k^\sharp = \Ev_E \circ (ah^\sharp + bk^\sharp)\;.$$

Hence, given $f,g \in \BornMeas(T,X)$, $a,b \in \RR$, taking $h := \delta_X \circ f$ and $k := \delta_X \circ g$ we have that
$$(a\delta_X \circ f + b\delta_X \circ g)^\sharp = a(\delta_X \circ f)^\sharp + b(\delta_X \circ g)^\sharp = aMf + bMg\;,$$
so we may compute, using the fact that $c$ is an $\MM$-homomorphism and a linear map, that
$$af^\sharp + bg^\sharp = ac \circ Mf + bc \circ Mg = c \circ (aMf + bMg) = c \circ (a\delta_X \circ f + b\delta_X \circ g)^\sharp$$
$$= (c \circ (a\delta_X \circ f + b\delta_X \circ g))^\sharp = (ac \circ \delta_X \circ f + bc \circ \delta_X \circ g)^\sharp = (af + bg)^\sharp\;.$$
\end{proof}

\section{Pettis integrals and $\MM$-algebras}

\begin{DefSub}
Let $X$ be a (real) Banach space.  Let $X^* := \RVect(\Born)(X,\RR)$ be the vector space of all bounded linear functionals on $X$.  The \textit{weak sigma-algebra} on $X$ is the initial sigma-algebra induced by the family of all bounded linear functionals $\varphi:X \rightarrow \RR$.  Given a measurable space $T$, we say that a function $f:T \rightarrow X$ is \textit{weakly measurable} if it is measurable with respect to the weak sigma-algebra on $X$, equivalently, if the composite $T \xrightarrow{f} X \xrightarrow{\varphi} \RR$ is measurable for all $\varphi \in X^*$.
\end{DefSub}

\begin{DefSub} \label{def:pint}
Let $X$ be a Banach space, $f:T \rightarrow X$ a weakly measurable function, and $\mu \in \SM T$ a finite signed measure on $T$.  We say that a vector $x \in X$ is a \textit{Pettis integral} of $f$ with respect to $\mu$ if
$$\forall \varphi \in X^* \;:\; \text{$\varphi \circ f$ is $\mu$-integrable, and $\varphi(x) = \int \varphi \circ f \;d\mu$}\;.$$
If such a Pettis integral exists, then, since the space of functionals $X^*$ separates points, this Pettis integral must be unique, and we denote it by $\pint f\;d\mu$ or $\pint_{t \in T} f(t)\;d\mu$.
\end{DefSub}

\begin{RemSub}
The defining property of the Pettis integral $\pint f \;d\mu$ in \ref{def:pint} requires exactly that
$$\forall \varphi \in X^* \;:\; \varphi\left(\pint_{t \in T} f(t)\;d\mu\right) = \int_{t \in T} \varphi(f(t))\;d\mu\;.$$
Compare this with the characterization of an $\MM$-homomorphism given in \ref{thm:homom}.
\end{RemSub}

\begin{DefSub}
We say that a Banach space $X$ \textit{has enough Pettis integrals} if for all bounded weakly measurable functions $f:T \rightarrow X$ and all finite signed measures $\mu \in \SM T$, there is a Pettis integral $\pint f\;d\mu$ in $X$.  Here, by \textit{bounded} we mean that $f$ has bounded image.
\end{DefSub}

\begin{PropSub}
For a Banach space $X$, the following are equivalent
\begin{enumerate}
\item $X$ has enough Pettis integrals;
\item $X$ has a Pettis integral $\pint f\;d\mu$ for each bounded weakly measurable function $f:T \rightarrow X$ and each nonnegative $\mu \in \SM T$.
\item $X$ has a Pettis integral $\pint \iota_{BX}\;d\mu$ of the inclusion $\iota_{BX}:B \hookrightarrow X$ for each bounded subset $B \subs X$ and each $\mu \in \SM B$.
\item $X$ has a Pettis integral $\pint \id_X \;d\mu$ of the identity map $\id_X:X \rightarrow X$ with respect to each finite signed measure of bounded support $\mu \in MX$.
\end{enumerate}
\end{PropSub}
\begin{proof}
The implications $(1)\Rightarrow(2)$ and $(1)\Rightarrow(3)$ are clear.  Also, $(2)\Rightarrow(1)$ since for any $\mu \in \SM T$, if there exist Pettis integrals $\pint f\;d\mu^+$ and $\pint f\;d\mu^-$, then $\pint f\;d\mu^+ - \pint f\;d\mu^-$ is a Pettis integral of $f$ with respect to $\mu$.  Next, $(3)\Rightarrow(4)$, since for any $\mu \in MX$, $\mu$ is supported by some bounded $B \subs X$ and hence for each $\varphi \in X^*$ we have that
$$\varphi\left(\pint \iota_{BX} d\mu_B\right) = \int \varphi \circ \iota_{BX}\;d\mu_B = \int \varphi \;d\mu$$
by \ref{thm:integ_of_restn}, so that $\pint \iota_{BX} d\mu_B$ is a Pettis integral of $\id_X$ with respect to $\mu$.  Lastly, suppose $(4)$.  Endow $X$ with the weak sigma-algebra, let $f:T \rightarrow X$ be a bounded weakly measurable function, and let $\mu \in \SM T$.  Endowing $T$ with the bornology consisting of all subsets of $T$, we have that $T$ is a bounded $\BornMeas$-object, $f:T \rightarrow X$ is a $\BornMeas$-morphism, and $\mu \in \SM T = MT$.  Hence we have $Mf(\mu) \in MX$, so there is a Pettis integral $\pint \id_X \;d\,Mf(\mu)$ in $X$, and this serves as a Pettis integral $\pint f\;d\mu$ since for each $\varphi \in X^*$ we have that
$$\varphi\left(\pint \id_X \;d\,Mf(\mu)\right) = \int \varphi\;d\,Mf(\mu) = \int \varphi \circ f\;d\mu$$
by Proposition \ref{thm:integ_of_composite}.
\end{proof}

\begin{ThmSub} \label{thm:enough_pints_implies_Malg}
Let $X$ be a Banach space with enough Pettis integrals.  Then $X$ is an $\MM$-algebra when we endow $X$ with the weak sigma-algebra, the norm bornology, and the structure map $c:MX \rightarrow X$ sending each $\mu \in MX$ to the Pettis integral $\pint \id_X \;d\mu$.
\end{ThmSub}
\begin{proof}
\textbf{(i).}  Firstly $c$ is measurable, since for every $\varphi \in X^*$, $\varphi$ is measurable and bornological, and the composite $MX \xrightarrow{c} X \xrightarrow{\varphi} \RR$
is the measurable map $\varphi^\sharp$ (\ref{thm:lifts_of_real_fns}), since for each $\mu \in MX$ we have $\varphi \circ c(\mu) = \varphi(\pint \id_X \;d\mu) = \int \varphi \;d\mu = \varphi^\sharp(\mu)$.

\paragraph*{(ii).} Next we prove that $c$ is bornological.  Consider any basic bounded subset $M(B,X,\gamma)$ of $MX$, where $B \subs X$ is bounded and $\gamma > 0$.  Then $\gamma B \subs X$ is bounded and hence is contained within some closed ball $B_r$ in $X$ of radius $r > 0$.  We shall show that $c(M(B,X,\gamma)) \subs B_r$.

Both the norm topology and the weak topology are admissible topologies on $X$ with respect to the dual pair $(X,X^*)$ (see, e.g., \cite{H}, \S 98), so since $B_r$ is convex and closed in the norm topology, $B_r$ is also closed in the weak topology (e.g., by \cite{H}, 98.1).  Furthermore, $B_r$ is absolutely convex, so $B_r$ is equal to its absolutely convex weakly-closed hull, which, by the Bipolar Theorem, is equal to the bipolar $B_r^{\circ\circ}$ (see, e.g., \cite{H}, \S 99).  Hence
$$B_r = B_r^{\circ\circ} = \{x_0 \in X \mid \forall \varphi \in B_r^\circ\;:\; |\varphi(x_0)| \lt 1\}$$
where
$$B_r^\circ = \{\varphi \in X^* \mid \forall x \in B_r \;:\; |\varphi(x)| \lt 1\}\;.$$

Now let $\mu \in M(B,X,\gamma)$.  To see that $c(\mu) \in B_r^{\circ\circ} = B_r$, consider any $\varphi \in B_r^\circ$.  Since $\mu$ is supported by $B$, we have by \ref{thm:integ_of_restn} that
$$\varphi(c(\mu)) = \varphi(\pint \id_X \;d\mu) = \int \varphi \;d\mu = \int \varphi \circ \iota_{BX} \;d\mu_B\;,$$
where $\iota_{BX}:B \hookrightarrow X$ is the inclusion.  But we have that $|\varphi| \lt \gamma^{-1}$ on $B$, since for any $x \in B$ we have $\gamma x \in \gamma B \subs B_r$ and hence $\gamma |\varphi(x)| = |\varphi(\gamma x)| \lt 1$, as $\varphi \in B_r^\circ$.  Therefore
$$
\begin{array}{lllllll}
 |\varphi(c(\mu))|& =   & |\int \varphi \circ \iota_{BX} \;d\mu_B| & \lt & \int |\varphi \circ \iota_{BX}| \;d|\mu_B| &   &   \\
                  & \lt & \gamma^{-1} |\mu_B|(B)           & =    & \gamma^{-1} ||\mu_B||                             &   &   \\
                  & =   & \gamma^{-1} ||\mu||              & =    & \gamma^{-1} \gamma                                & = & 1\;,
\end{array}
$$
using Proposition \ref{thm:total_var_of_restn}.

\paragraph*{(iii).}  To see that $(X,c)$ satisfies the unit law $c \circ \delta_X = 1_X$, let $x \in X$.  For each $\varphi \in X^*$,
$$\varphi \circ c \circ \delta_X(x) = \varphi(\pint \id_X \;d\delta_x) = \int \varphi \;d\delta_x = \varphi(x)\;.$$
Hence, since the $\varphi \in X^*$ separate points, $c \circ \delta_X(x) = x$.

\paragraph*{(iv).} In order to prove that $(X,c)$ satisfies the associativity law $c \circ Mc = c \circ \kappa_X$, let $\M \in MMX$.  For each $\varphi \in X^*$, we have that
$$\varphi \circ c \circ \kappa_X(\M) = \varphi\left(\pint \id_X \;d\,\kappa_X(\M)\right) = \int \varphi \;d\,\kappa_X(\M) = \int \varphi^\sharp \;d\M$$
by Proposition \ref{thm:lifts_of_real_fns}, whereas
$$\varphi \circ c \circ Mc(\M) = \varphi\left(\pint \id_X \;d\,Mc(\M)\right) = \int \varphi \;d\,Mc(\M) = \int \varphi \circ c \;d\M$$
by Proposition \ref{thm:integ_of_composite}.  But as noted in (i), $\varphi \circ c = \varphi^\sharp$, so these two real numbers are equal.  Hence, since the $\varphi \in X^*$ separate points, the result is established.
\end{proof}

\begin{CorSub}
Any Banach space $X$ that is separable or reflexive has enough Pettis integrals and hence is an $\MM$-algebra when endowed with the bornology, sigma-algebra, and structure map of Theorem \ref{thm:enough_pints_implies_Malg}.
\end{CorSub}
\begin{proof}
\textbf{(i).}  It is well-known that any separable Banach space $X$ has enough Pettis integrals.  For example, given a bounded weakly measurable $f:T \rightarrow X$ and any nonnegative finite measure $\mu$ on $T$, $f$ is in particular \textit{scalarly bounded}, in the terminology of \cite{Pl}, so since $X$ is separable we deduce by Theorem 1 of \cite{Pl} that $f$ is \textit{Pettis integrable} in the sense employed there, which implies in particular that a Pettis integral $\pint f \;d\mu$ exists.

\paragraph*{(ii).}  For a reflexive Banach space $X$, if a function $f:T \rightarrow X$ is \textit{Dunford integrable} with respect to a finite nonnegative measure $\mu$ on $T$, meaning that $f$ is weakly measurable and each $\varphi \circ f$ is $\mu$-integrable ($\varphi \in X^*$), then $f$ is Pettis integrable and, in particular, there is a Pettis integral $\pint f \;d\mu$; see \cite{DU}, \S II.3.  Hence $X$ has enough Pettis integrals, since a bounded weakly measurable function $f:T \rightarrow X$ is Dunford integrable with respect to any finite nonnegative measure $\mu$ on $T$.
\end{proof}

\begin{RemSub}
Reflexive Banach spaces include all Hilbert spaces and all spaces $L^p(\mu)$ with $1 < p < \infty$ (see, e.g., \cite{Me}).
\end{RemSub}
\bibliographystyle{amsplain}
\bibliography{integ}

\providecommand{\bysame}{\leavevmode\hbox to3em{\hrulefill}\thinspace}
\providecommand{\MR}{\relax\ifhmode\unskip\space\fi MR }
\providecommand{\MRhref}[2]{%
  \href{http://www.ams.org/mathscinet-getitem?mr=#1}{#2}
}
\providecommand{\href}[2]{#2}
\begin{thebibliography}{10}

\bibitem{AHS}
J.~Ad{\'a}mek, H.~Herrlich, and G.~E. Strecker, \emph{Abstract and concrete
  categories}, John Wiley \& Sons Inc., 1990.

\bibitem{ARV}
J.~Ad{\'a}mek, J.~Rosick{\'y}, and E.~M. Vitale, \emph{Algebraic theories},
  Cambridge University Press, 2011, A categorical introduction to general
  algebra, With a foreword by F. W. Lawvere.

\bibitem{Bo}
V.~I. Bogachev, \emph{Measure theory. {V}ol. {I}, {II}}, Springer-Verlag, 2007.

\bibitem{DU}
J.~Diestel and J.~J. Uhl, Jr., \emph{Vector measures}, American Mathematical
  Society, 1977.

\bibitem{D}
E.{-}E. Doberkat, \emph{Eilenberg-{M}oore algebras for stochastic relations},
  Inform. and Comput. \textbf{204} (2006), no.~12, 1756--1781.

\bibitem{EM}
S.~Eilenberg and J.~C. Moore, \emph{Adjoint functors and triples}, Illinois J.
  Math. \textbf{9} (1965), 381--398.

\bibitem{FK}
A.~Fr{\"o}licher and A.~Kriegl, \emph{Linear spaces and differentiation
  theory}, John Wiley \& Sons Ltd., 1988.

\bibitem{G}
M.~Giry, \emph{A categorical approach to probability theory}, Lecture Notes in
  Math., vol. 915, Springer, 1982, pp.~68--85.

\bibitem{H}
H.~G. Heuser, \emph{Functional analysis}, John Wiley \& Sons Ltd., 1982,
  Translated from the German by John Horv{\'a}th.

\bibitem{HN}
H.~Hogbe-Nlend, \emph{Bornologies and functional analysis}, North-Holland
  Publishing Co., 1977.

\bibitem{K2}
A.~Kock, \emph{Calculus of extensive quantities},
  \verb+http://arxiv.org/abs/1105.3405+, 2011.

\bibitem{K1}
\bysame, \emph{Monads and extensive quantities},
  \verb+http://arxiv.org/abs/1103.6009+, 2011.

\bibitem{LS}
S.~Lack and R.~Street, \emph{The formal theory of monads. {II}}, J. Pure Appl.
  Algebra \textbf{175} (2002), no.~1-3, 243--265.

\bibitem{La}
F.~W. Lawvere, \emph{Functorial semantics of algebraic theories}, Dissertation,
  Columbia University, New York. Available in: \emph{Repr. Theory Appl. Categ.}
  \textbf{5} (2004), 1963.

\bibitem{Li1}
F.~E.~J. Linton, \emph{Some aspects of equational categories}, Proc. {C}onf.
  {C}ategorical {A}lgebra ({L}a {J}olla, {C}alif., 1965), Springer, 1966,
  pp.~84--94.

\bibitem{Li2}
\bysame, \emph{Triples vs. theories}, Notices Amer. Math. Soc. \textbf{13}
  (1966), 227, preliminary report.

\bibitem{Li3}
\bysame, \emph{An outline of functorial semantics}, Sem. on {T}riples and
  {C}ategorical {H}omology {T}heory ({ETH}, {Z}\"urich, 1966/67), Springer,
  1969, Available in: \emph{Repr. Theory Appl. Categ.} \textbf{18} (2008),
  pp.~7--52.

\bibitem{MW}
F.~Marmolejo and R.~J. Wood, \emph{Monads as extension systems---no iteration
  is necessary}, Theory Appl. Categ. \textbf{24} (2010), no.~4, 84--113.

\bibitem{Me}
R.~E. Megginson, \emph{An introduction to {B}anach space theory},
  Springer-Verlag, 1998.

\bibitem{PT}
M.~C. Pedicchio and W.~Tholen (eds.), \emph{Categorical foundations}, Cambridge
  University Press, 2004, Special topics in order, topology, algebra, and sheaf
  theory.

\bibitem{Pe}
B.~J. Pettis, \emph{On integration in vector spaces}, Trans. Amer. Math. Soc.
  \textbf{44} (1938), no.~2, 277--304.

\bibitem{Pl}
G.~Plebanek, \emph{On {P}ettis integrals with separable range}, Colloq. Math.
  \textbf{64} (1993), no.~1, 71--78.

\bibitem{Se}
Z.~Semadeni, \emph{Monads and their {E}ilenberg-{M}oore algebras in functional
  analysis}, Queen's University, 1973, Queen's Papers in Pure and Applied
  Mathematics, No. 33.

\bibitem{St}
R.~Street, \emph{The formal theory of monads}, J. Pure Appl. Algebra \textbf{2}
  (1972), no.~2, 149--168.

\bibitem{Sw1}
T.~\'Swirszcz, \emph{Monadic functors and categories of convex sets}, Preprint
  no. 70, Institute of Mathematics, Polish Academy of Sciences.

\bibitem{Sw2}
\bysame, \emph{Monadic functors and convexity}, Bull. Acad. Polon. Sci. S\'er.
  Sci. Math. \textbf{22} (1974), no.~1, 39--42.

\bibitem{T}
M.~Talagrand, \emph{Pettis integral and measure theory}, Mem. Amer. Math. Soc.
  \textbf{51} (1984), no.~307.

\end{thebibliography}

\end{document}